\documentclass[12pt,twoside]{article}
\usepackage{amsfonts}
\usepackage{amsmath}

\input{tcilatex}
\begin{document}

\date{}

\centerline{\bf }

\centerline{}

\centerline{}

\centerline {\Large{\bf Global Existence of Solutions for Some}}

\centerline{}

\centerline{\Large{\bf Coupled Systems of Reaction-Diffusion}}

\centerline{}

\centerline{\bf {Abdelmalek Salem}}

\centerline{}

\centerline{Department of mathematics}

\centerline{University of Tebessa, 12002 Algeria}

\centerline{a.salem@gawab.com}

\centerline{}

\centerline{\bf {Youkana Amar}}

\centerline{}

\centerline{Department of Mathematics}

\centerline{University of Batna, 5000 Algeria}

\centerline{youkana-amar@yahoo.fr}

\begin{abstract}
The aim of this work is to study the global existence of solutions for some
coupled systems of reaction diffusion which describe the spread within a
population of infectious disease. We consider a triangular matrix diffusion
and we show that we can prove global existence of classical solutions for
the nonlinearities of weakly exponential growth.
\end{abstract}

\newtheorem{Theorem}{\quad Theorem}[section]

\newtheorem{Definition}[Theorem]{\quad Definition}

\newtheorem{Corollary}[Theorem]{\quad Corollary}

\newtheorem{Lemma}[Theorem]{\quad Lemma}

\newtheorem{Example}[Theorem]{\quad Example}



\section{Introduction}

In this work we shall be concerned with a reaction-diffusion system of the
form:%
\begin{equation}
\left\{
\begin{array}{l}
\frac{du}{dt}-a\Delta u=\Lambda -\lambda \left( t\right) f(u,v)-\mu
u\,\,\,\,\ \ \ \ \ \ \ \ \ \ \ \text{in }%
\mathbb{R}
^{+}\times \Omega \\
\frac{dv}{dt}-b\Delta u-d\Delta v=\lambda \left( t\right) f(u,v)-\mu v\,\,\
\ \ \,\ \ \ \ \ \text{in }%
\mathbb{R}
^{+}\times \Omega%
\end{array}%
\right.  \tag{1.1}
\end{equation}%
with boundary conditions%
\begin{equation}
\frac{\partial u}{\partial \eta }=\frac{\partial v}{\partial \eta }%
=0\,\,\,\,\,\text{on\ }%
\mathbb{R}
^{+}\times \partial \Omega  \tag{1.2}
\end{equation}%
and the nonnegative continuous and bounded initial data
\begin{equation}
u\left( 0,x\right) =u_{0}\left( x\right) ,\,\,\,\,\,\,\,\,v\left( 0,x\right)
=v_{0}\left( x\right) \,\ \ \text{in }\overset{-}{\Omega }  \tag{1.3}
\end{equation}%
where\ $\Omega $\ is a bounded domain of class $C^{1}$ in\ $%
\mathbb{R}
^{n},$ with boundary $\partial \Omega ,\,\frac{\partial }{\partial \eta }$
is the outward normal derivative to $\partial \Omega $.

The constants $a,b,d,\ \Lambda ,\ \mu $ are such that%
\begin{equation}
a>0,\text{ \ \ }b>0,\ \ d-a\geq b,\ \ \ \mu
>0,\text{ \ and \ }\Lambda \geq 0\text{.}  \tag{H.1}
\end{equation}

We assume that $t\rightarrow \lambda \left( t\right) $ is a nonnegative and
bounded function in $C\left(
\mathbb{R}
^{+}\right) $ with $0\leq \lambda \left( t\right) \leq \widehat{\lambda }$\
and nonlinearity $f\,$\ is a nonnegative continuously differentiable
function on $\left( 0,+\infty \right) \times \left( 0,+\infty \right) \ $%
satisfying%
\begin{equation}
f\left( 0,\eta \right) =0\ \ \text{for\ all}\ \eta \ \ \text{in}\
\mathbb{R}
^{+}\ \text{and}\ \underset{\eta \rightarrow +\infty }{\lim }\frac{\log
(1+f(.,\eta ))}{\eta }=0\text{.}  \tag{H.2}
\end{equation}

The reaction-diffusion system $(1.1)-(1.3)$ may be viewed as a diffusive
epidemic model where $u$ and $v$ represent the nondimensional population
densities of susceptibles and invectives, respectively. We can \ consider
the system $(1.1)-(1.3)$ as a model describing the spread of an infection
disease (such as AIDS for instance) within a population assumed to be
divided into the susceptible and infective classes as precised (for further
motivation see for instance \cite{BC, CCH} and the references therein).

We note that in the case $\Lambda =\mu =0,$ this problem corresponds to the
problem of R. H.Martin and much works have been done in the literature and
positive answers have been under different forms for subgrowth nonlinearity
(se for instance \cite{A, MA, HY, HMP, K, M}).

We mentioned that in the case $\Lambda >0,~$it is not obvious to obtain
global existence of solutions when the nonlinearity are at most polynomial
growth and the methods as above can not be applied.

Our main purpose is to study the global existence of solutions of the system
$(1.1)-(1.3)$ with nonlinearities of weakly exponential growth.

\section{Preliminaries results}

\subsection{Local existence of solutions.}

We denote by $W^{m,p}(\Omega )$, the Sobolev space of order $m\geq 0$\ for\ $%
1\leq p\leq +\infty \ $and$\ $by$\ C(\overset{-}{\Omega })$ the Banach space
of continuous~functions on $\overset{-}{\Omega }$.

We can convert equations $(1.1)-(1.3)$ to an abstract first order system in
the Banach space $X=C(\overset{-}{\Omega })\times C(\overset{-}{\Omega })$\
of the form%
\begin{equation*}
\left\{
\begin{array}{c}
\frac{dU(t)}{dt}=\breve{A}U(t)+F(U(t)) \\
U(0)=U_{0}\in X%
\end{array}%
\ \ \ t>0,\right.
\end{equation*}%
where%
\begin{equation*}
\tilde{A}:D_{\infty }(A)\times D_{\infty }(B)\rightarrow X,
\end{equation*}%
with%
\begin{equation*}
D_{\infty }(A)=D_{\infty }(B)=\left\{ z\in W^{2,p}(\Omega )\ \text{for all }%
p>n,\ \Delta z\in C(\overset{-}{\Omega }),\ \frac{\partial z}{\partial \eta }%
=0\right\} ,
\end{equation*}%
\begin{equation*}
\tilde{A}U=(Au,Bv)
\end{equation*}%
and%
\begin{equation*}
F(U)=(\Lambda -\lambda f\left( u,v\right) ;\lambda f\left( u,v\right) )\text{%
.}
\end{equation*}

It is clear that from the general theory of semigroup we deduce the
existence of an unique local classical solution in some interval $]0,T^{\ast
}[$, where $T^{\ast }$ is the eventual blowing-up time in $L^{\infty
}(\Omega )$ (see for example Henry \cite{H} or Pazy \cite{P}).

\subsection{Positivity of solutions.}

From the nonnegativity of the initial data $u_{0}$ on $\Omega ,$ one~easily
deduce from the maximum applied to the first equation of $(1.1)$ that the
component $u\ $remains nonnegative and bounded on ($0,T^{\ast })\times
\Omega $ so that
\begin{equation*}
0\leq u(t,x)\leq K:=\max (\left\Vert u_{0}\right\Vert _{\infty },\frac{%
\Lambda }{\mu })\text{.}
\end{equation*}

In order to obtain the positivity of $v$\ we assume that%
\begin{equation}
\left\Vert u_{0}\right\Vert _{\infty }\leq \frac{\Lambda }{\mu }\text{.}
\tag{H.3}
\end{equation}

\begin{lemma}
Let $(u,v)$\ be the solution of $(1.1)-(1.3)$.\newline
If\ the initial data $v_{0}$ satisfies the condition%
\[
v_{0}\geq \frac{b}{d-a}(\frac{\Lambda }{\mu }-\left\Vert u_{0}\right\Vert
_{\infty })
\]%
\ then for\ all$\ (t,x)\in (0,T^{\ast })\times \Omega \ $we\ have
\[
v(t,x)\geq \frac{b}{d-a}(\frac{\Lambda }{\mu }-u(t,x))\text{.}
\]
\end{lemma}

\begin{proof}
Let us consider $w=v-\frac{b}{d-a}(\frac{\Lambda }{\mu }-u)~$and in this way
the system $(1.1)-(1.3)$ may be equivalent to the system%
\begin{equation}
\left\{
\begin{array}{l}
\frac{du}{dt}-a\Delta u=\Lambda -\lambda f(u,v)-\mu u \\
\frac{dw}{dt}-d\Delta w=(1-\frac{b}{d-a})\lambda f(u,v)-\mu w%
\end{array}%
\right. \text{in }%
\mathbb{R}
^{+}\times \Omega  \tag{2.1}
\end{equation}%
with%
\begin{equation}
\left\{
\begin{array}{c}
u(0,.)=u_{0}(.)\geq 0 \\
w(0,.)=w_{0}(.)\geq 0%
\end{array}%
\right. \text{in}\ \Omega \text{.}  \tag{2.2}
\end{equation}

Using $(H.1)$ and maximum principle to the second equation of the system $%
(2.1)$ we obtain%
\[
w(t,x)\geq 0\ \forall (t,x)\in (0,T^{\ast })\times \Omega
\]%
that means%
\begin{equation}
v(t,x)\geq \frac{b}{d-a}(\frac{\Lambda }{\mu }-u(t,x))\geq 0\ \text{for\ all}%
\ (t,x)\in (0,T^{\ast })\times \Omega \text{.}  \tag{2.3}
\end{equation}
\end{proof}

Now we are able to state our main result.

\section{Statement and proof of the main results}

Using the idea of Haraux and Youkana \cite{HY}, let us consider the
functional%
\begin{equation*}
J(t)=\int_{\Omega }\left( 1+\delta (1+u+u^{2}\right) e^{\varepsilon w}dx
\end{equation*}

where\ $\delta \ $and $\varepsilon \ $are constants such that%
\begin{equation}
0<\delta \leq \min \left( \frac{\mu }{2\Lambda \left( 1+2K\right) }%
,2\left( \frac{2\sqrt{ab}}{a+b}\cdot \frac{1}{1+2K}\right) ^{2}\right)
\tag{3.1}
\end{equation}%
and%
\begin{equation}
0<\varepsilon \leq \frac{\delta }{1+\delta \left( K+K^{2}\right) }\min
\left( 1,\frac{d-a}{b}\right) \text{.}  \tag{3.2}
\end{equation}

Our main results are as follows:

\begin{theorem}
Let\thinspace $\ \left( u,v\right) \,$be a solution of $\left( 1.1\right)
-\left( 1.3\right) $\ on $(0,T^{\ast }),\,$then there exist a positive
constant $\gamma \,$such that\thinspace \thinspace \thinspace for all
\thinspace $t\in (0,T^{\ast })\ $the\ functional%
\begin{equation}
J(t)=\int_{\Omega }\left( 1+\delta (1+u+u^{2}\right) e^{\varepsilon
v}dx \tag{3.3}
\end{equation}

satisfies the\ relation%
\begin{equation}
\frac{d}{dt}J(t)\leq -\frac{\mu }{2}J(t)+\gamma \text{.}  \tag{3.4}
\end{equation}
\end{theorem}

\begin{corollary}
Under the hypothesis $(H.1)-(H.3)$, the solutions of the parabolic system $%
(1.1)-(1.3)$\ with nonnegative and bounded initial data $u_{0},v_{0}$ are
global and uniformly bounded in $(0,+\infty )\times \Omega $.
\end{corollary}

\begin{proof}[Proof of the theorem]
By simple use of Green's formula and from equations $(1.1)-(1.3)$ we have
for all $t\in (0,T^{\ast })$%
\[
\frac{d}{dt}J(t)=G+H_{1}+H_{2}+H_{3}
\]%
where%
\[
\left.
\begin{array}{l}
G=-\delta \int_{%
\Omega
}\left[ 2a+b\varepsilon \left( 1+2u\right) \right] e^{\varepsilon
v}\left\vert \nabla u\right\vert ^{2}dx \\
\text{ \ \ \ \ \ \ }-\varepsilon \int_{%
\Omega
}\left[ \left( a+d\right) \delta \left( 1+2u\right) +b\varepsilon
\left( 1+\delta \left( u+u^{2}\right) \right) \right] e^{\varepsilon
v}\nabla
u\nabla vdx \\
\text{ \ \ \ \ \ \ }-d\varepsilon ^{2}\int_{%
\Omega
}\left[ 1+\delta \left( u+u^{2}\right) \right] e^{\varepsilon
v}\left\vert
\nabla v\right\vert ^{2}dx%
\end{array}%
\right.
\]%
and%
\[
\left.
\begin{array}{l}
H_{1}=\int_{\Omega }(\Lambda \frac{\delta (1+2u)}{1+\delta (u+u^{2})}-\mu u%
\frac{\delta (1+2u)}{1+\delta (u+u^{2})}-\mu )(1+\delta
(u+u^{2})e^{\varepsilon v}dx \\
H_{2}=\int_{\Omega }\left( \varepsilon -\frac{\delta \left( 1+2u\right) }{%
1+\delta \left( u+u^{2}\right) }\right) \lambda f\left( u,v\right) \left(
1+\delta \left( u+u^{2}\right) \right) e^{\varepsilon v}dx \\
H_{3}=\mu \int_{\Omega }(1-\varepsilon v)e^{\varepsilon v}(1+\delta
(u+u^{2})dx\text{.}%
\end{array}%
\right.
\]

We observe that G involves a quadratic form with respect to $\nabla u$ and $%
\nabla v$%
\[
\left.
\begin{array}{l}
Q=\delta \left( 2a+b\varepsilon \left( 1+2u\right) \right) e^{\varepsilon
v}\left\vert \nabla u\right\vert ^{2} \\
\text{ \ \ \ \ \ \ }+\varepsilon \left[ \left( a+d\right) \delta \left(
1+2u\right) +b\varepsilon \left( 1+\delta \left( u+u^{2}\right) \right) %
\right] e^{\varepsilon v}\nabla u\nabla v \\
\text{ \ \ \ \ \ \ }+d\varepsilon ^{2}\left( 1+\delta \left( u+u^{2}\right)
\right) e^{\varepsilon v}\left\vert \nabla v\right\vert ^{2}%
\end{array}%
\right.
\]%
and the discriminant%
\[
\left.
\begin{array}{c}
D=\left[ \varepsilon \left[ \left( a+d\right) \delta \left( 1+2u\right)
+b\varepsilon \left( 1+\delta \left( u+u^{2}\right) \right) \right] \right]
^{2}- \\
\text{ \ \ \ \ }4\left[ \delta \left( 2a+b\varepsilon \left( 1+2u\right)
\right) \right] \left[ d\varepsilon ^{2}\left( 1+\delta \left(
u+u^{2}\right) \right) \right] ,%
\end{array}%
\right.
\]

may be non-positive since the constants $\delta $ and $\varepsilon $
satisfies $(3.1)\ $and$\ (3.2)$.

Consequently%
\[
G\leq 0\ \text{for all }t\ \text{\ in }\left( 0,T^{\ast }\right) \text{.}
\]

Concerning the terms $H_{i},\ i=1,2,3\ $we have\ again from $(3.1)-(3.2)$
where now%
\begin{equation}
0<\delta \leq \frac{\mu }{2\Lambda \left( 1+2K\right) },\ \ 0<\varepsilon
\leq \frac{\delta }{1+\delta \left( K+K^{2}\right) },  \tag{3.5}
\end{equation}%
one checks that%
\begin{equation}
\Lambda \frac{\delta \left( 1+2u\right) }{1+\delta \left( u+u^{2}\right) }%
-\mu u\frac{\delta \left( 1+2u\right) }{1+\delta \left( u+u^{2}\right) }-\mu
\leq 2\Lambda \delta \left( 1+2K\right) -\mu \leq -\frac{\mu }{2},  \tag{3.6}
\end{equation}%
and
\begin{equation}
\varepsilon -\frac{\delta \left( 1+2u\right) }{1+\delta \left(
u+u^{2}\right) }\leq \varepsilon -\frac{\delta }{1+\delta \left(
K+K^{2}\right) }\leq 0,  \tag{3.7}
\end{equation}%
from we deduce%
\[
H_{1}\leq -\frac{\mu }{2}J(t)
\]%
and%
\[
H_{2}\leq 0\text{.}
\]

Now for the term $H_{3},$one observe that the function%
\[
\pi :\eta \rightarrow \left( 1-\varepsilon \eta \right) e^{\varepsilon \eta
},
\]%
is bounded on $%
\mathbb{R}
^{+}$. Indeed, one has%
\[
\frac{d\pi }{d\eta }\left( \eta \right) =-\varepsilon ^{2}\eta
e^{\varepsilon \eta }\leq 0,
\]%
so that $\pi $ is nonincreasing in $[0,+\infty \lbrack $ and%
\[
\underset{\eta \geq 0}{\max }\left( 1-\varepsilon \eta \right)
e^{\varepsilon \eta }=1\text{.}
\]

In this way one can deduce that there a positive constant $\gamma $%
\[
\gamma =\mu \left( 1+\delta \left( K+K^{2}\right) \right) \left\vert \Omega
\right\vert
\]%
such that
\begin{equation}
H_{3}\leq \gamma \text{.}  \tag{3.8}
\end{equation}%
Aditioning~$G,$ $H_{1},\ H_{2}\ $and$\ H_{3}$ we get%
\begin{equation}
\frac{d}{dt}J(t)=G+H_{1}+H_{2}+H_{3}\leq -\frac{\mu }{2}J\left(
t\right) +\gamma \text{.}  \tag{3.9}
\end{equation}%
Thus, the proof of theorem is complete.
\end{proof}

Now we are able to prove the global existence for the solutions of $%
(1.1)-(1.3)$.

\begin{proof}[Proof of the corollary]
From $(H.2)$ we see that there exist a positive consant $C$\ such that%
\begin{equation}
1+f(.,v)\leq Ce^{\frac{v}{n}},\ \forall v\geq 0.  \tag{3.10}
\end{equation}%
where $\varepsilon \ $\ is chosen\ as\ in\ $(3.2)$. By virtue of $(3.9)$, it
is seen that%
\[
J(t)\leq C,\ \ \ \ \ \ \ \forall t\in (0,T^{\ast })
\]%
and it follows in particular that%
\[
f(.,v)\in L^{\infty }((0,T^{\ast });L^{n}(\Omega ))\text{.}
\]

By the regularizing effect of the heat equation \cite{H, HK} we conclude that%
\[
u\in L^{\infty }((0,T^{\ast });\text{ }L^{\infty }(\Omega ))\text{.}
\]

Finally the solutions of the system $(1.1)-(1.3)$~are global and
uniformly bounded on $(0,+\infty )\times \Omega $\ and the corollary
is completely proved.
\end{proof}

\textbf{ACKNOWLEDGEMENTS.} We would like to thank a lot Prof. M.
KIRANE (Universit\'{e} de la Rochelle, FRANCE)for his help and
guidance..

\end{document}